\documentclass[12pt]{amsart}
\usepackage{amssymb, url}
\usepackage[top=1in, left=1in, right=1in, bottom=1in]{geometry}
\newtheorem{theorem}{Theorem}[section]

\newtheorem{lemma}[theorem]{Lemma}
\newtheorem{corollary}[theorem]{Corollary}

\theoremstyle{definition}
\newtheorem{definition}[theorem]{Definition}

\newtheorem{example}[theorem]{Example}

\newcommand{\e}{\epsilon}
\newcommand{\A}{\mathcal{A}}
\newcommand{\B}{\mathcal{B}}
\newcommand{\C}{\mathbb{C}}
\renewcommand{\H}{\mathcal{H}} 
\newcommand{\I}{\mathcal{I}} 
\newcommand{\J}{\mathcal{J}} 
\newcommand{\M}{\mathcal{M}}
\newcommand{\N}{\mathcal{N}} 
\newcommand{\NN}{\mathbb{N}} 
\renewcommand{\P}{\mathcal{P}}
\renewcommand{\S}{\mathcal{S}}
\newcommand{\T}[1]{\mathcal{T}(#1)} 
\newcommand{\bh}{B(\H)} 
\newcommand{\condex}[1]{\Delta(#1)}
\newcommand{\essnorm}[1]{\| #1 \|_\text{ess}}
\newcommand{\idealof}{\triangleleft}
\newcommand{\jmin}{\J_\text{min}}
\newcommand{\kernel}[1]{\mathop{ker}(#1)}
\newcommand{\masa}{masa}
\newcommand{\onenorm}[1]{\| #1 \|_1}
\newcommand{\rank}[1]{\mathop{rank}(#1)}
\newcommand{\tr}[1]{\mathop{tr}(#1)}

\begin{document}

\title{The Maximal Two-Sided Ideals of Nest Algebras}
\author{John Lindsay Orr}
\address{Google Inc., 1600 Amphitheatre Parkway, Mountain View, CA 94043, USA}
\email{jorr@google.com}

\begin{abstract}
  We give a necessary and sufficient criterion for an operator in a nest algebra
  to belong to a proper two-sided ideal of that algebra. Using this result, we
  describe the strong radical of a nest algebra, and give a general description
  of the maximal two-sided ideals.
\end{abstract}

\maketitle

\section{Introduction}

The maximal two-sided ideals -- in this paper, all ideals will be assumed
two-sided -- of the algebra $T_n$ of $n \times n$ upper triangular matrices are all,
trivially, of the following form: the set of upper triangular matrices $(t_{ij})$
vanishing on some fixed diagonal entry. There is a natural extension of this to
bounded operators on separable infinite-dimensional Hilbert space. Let $\A$ be the
set of upper triangular operators with respect to a fixed orthonormal basis, let
$\M$ be the abelian C$^*$-algebra of diagonal operators, and let $\J$ be a
maximal ideal of $\M$. Then the set of upper triangular operators whose diagonal
part belongs to $\J$ is easily seen to be a maximal ideal of $\A$ (see
Example~\ref{order-n-nest} below). However, it
has been a tantalizing open question for a number of years whether \textit{all}
the maximal ideals of $\A$ are of this form
\cite[Section 8]{DavidsonHarrisonOrr:EpNeAl}.

In this paper we shall use Marcus,
Spielman, and Srivastava's recent solution of the Kadison Singer Problem
\cite{MarcusSielmanSrivastava:MiChPoKaSiPr} to
answer this question affirmatively. In fact we will go further. The algebra
$\A$ is a nest algebra, a class of algebras generalizing $T_n$ to infinite
dimensions, and, building on work in \cite{Orr:MaIdNeAl} and
\cite{Orr:TrAlIdNeAl} we shall give a description of the maximal ideals of all
nest algebras (Corollary~\ref{characterize-the-max-ideals}). The main tool to
do this will be a necessary and sufficient criterion for an operator in a  nest
algebra to belong to a proper ideal (Theorem~\ref{belongs-to-max-ideal}).

Nest algebras as a class were first introduced in the '60's by Ringrose
\cite{Ringrose:OnSoAlOp} and a rich structural theory was developed over
the next three decades (see the authoritative introduction \cite{Davidson:NeAl}).
Many authors have studied the ideal structure of nest algebras
\cite{Arveson:InPrNeAl, Erdos:OnSoIdNeAl, Orr:StIdNeAl, Orr:StIdNeAl2,
  Orr:MaIdNeAl, Dai:NoPrBiNeAl, DavidsonLeveneMarcouxRadjavi:ToStRaNoSeOpAl},
starting with Ringrose's description of the Jacobson Radical
\cite{Ringrose:OnSoAlOp}. Davidson's Similarity Theorem
\cite{Davidson:SiCoPeNeAl} provided a powerful tool to investigate the structure
of nest algebras \cite{LarsonPitts:IdNeAl, Orr:TrAlIdNeAl} and the continuous
nest algebras proved most amenable to this treatment. As a result there is
detailed information known about the ideal structure of continuous nest algebras
\cite{Orr:StIdNeAl, Orr:StIdNeAl2, Orr:MaIdNeAl} including a complete
description of the maximal ideals. However deeper understanding of the ideal structure of
other nest algebras has been blocked by our inability to answer the question
raised in the first paragraph and which this paper answers: what are the
maximal ideals of the infinite upper triangular operators?

In a unital algebra every proper two-sided ideal is contained in
a maximal proper two-sided ideal. The intersection of all maximal two-sided
ideals is called the \textit{strong radical} of the algebra. The related
Jacobson Radical, which is the intersection of the maximal left ideals, was
characterized by Ringrose in \cite{Ringrose:OnSoAlOp}, and the strong radical of
a continuous nest algebras was described in \cite{Orr:MaIdNeAl}. We shall shall
characterize the strong radical of a general nest algebra in
Theorem~\ref{strong-radical}.

We briefly remind the reader of a few facts about nest algebras. For a full
background see \cite{Davidson:NeAl}.
Let $\H$ be a separable Hilbert space (in this paper, we assume all our Hilbert
spaces are separable). A \textit{nest} $\N$ is a linearly ordered set of
projections on $\H$ which contains $0$ and $I$ and is $w^*$-closed
(equivalently, order-complete). The \textit{nest algebra} $\T{\N}$ of a nest
$\N$ is the set of bounded operators leaving invariant the ranges of the
projections in $\N$. An \textit{interval} of $\N$ is the difference
of two projections $N>M$ in $\N$. Minimal intervals are called \textit{atoms}
and the atoms (if there are any) are pairwise orthogonal. If the join of the
atoms is $I$ the nest is called \textit{atomic}; if there are no atoms it is
called \textit{continuous}.

Let $P_a$ be the sum of all the atoms in $\N$ and $P_c = I - P_a$. Then
$\N_a := P_a\N$ is an atomic nest and $\N_c = P_c\N$ is a continuous nest.
By slight abuse of notation we shall consider $\T{\N_a} = P_a \T{\N}P_a$
and $\T{\N_c} = P_c \T{\N}P_c$ to be subalgebras of $\T{\N}$.

\section{Examples}

In this section we shall present two examples of the different appearance of
maximal ideals in atomic and continuous nest algebras. The final description of
maximal ideals (Corollary~\ref{characterize-the-max-ideals}) will be blend of
these two forms.

\begin{example}\label{order-n-nest}
  Let $e_1, e_2, \ldots$ be the standard basis for $l^2(\NN)$. The set $\A$ of
  upper triangular operators is a nest algebra with the nest projections being
  the projections onto the first $n$ basis elements ($n=1,2,\ldots$), together with $I$.
  The algebra of diagonal operators is identified with $l^\infty(\NN)$ and the set of all
  strictly upper triangular operators, $\S$, is an ideal. If $\J$ is a maximal
  ideal of the diagonal algebra (which corresponds to a point in the
  Stone-\v{C}ech compactification of $\NN$) then $\J + \S$ is a maximal idea
  of $\A$.
\end{example}

Our main blocking problem, which the positive answer to the Kadison-Singer
Problem resolves, is whether \textit{all} maximal ideals of the algebra in
Example~\ref{order-n-nest} are of this form. If every maximal ideal of $\A$
contains $\S$ then then the answer is yes. So conversely, consider for a moment
the possibility that there is a maximal ideal $\J$ which does not contain $\S$.
It follows that $I \in \A = \J + \S$ and so there is an strictly upper
triangular operator $X$ such that $I+X \in \J$. It is very counter-intuitive
to imagine that this is possible, and the solution to the Kadison-Singer Problem
in fact shows that it is impossible, a fact on which the results of this paper
rest.

In the case of continuous nest algebras, the maximal ideals are fully described
by the results of \cite{Orr:MaIdNeAl}. The following example summarizes
\cite[Proposition 2.6]{Orr:MaIdNeAl}.

\begin{example}
  Let $\N$ be a continuous nest. Then $\N$ has an absolutely continuous
  parameterization by the unit interval \cite{Erdos:UnInNe} as $(N_t)_{t=0}^1$.
  For $X \in \T{\N}$ define the following
  parameterized seminorm which measures the size of $X$ asymptotically close to
  the diagonal:
  \[
    j_X(t) := \inf_{s<t<u} \|(N_u - N_s)X(N_u - N_s)\|
  \]
  For $a>0$ let $s_a(X) := \{ t : j_X(t) < \e \}$. Let $C$ be a collection of
  open subsets of $(0,1)$ which is closed under intersections; which contains
  an open set which has an element of $C$ as a subset; and which is maximal with
  respect to the first two properties. Then
  \[
    \{ X \in \T{\N} : s_a(X) \in C \text{ for all $a > 0$} \}
  \]
  is a maximal ideal of $\T{\N}$ and every maximal ideal of $\T{\N}$ is of this
  form.
\end{example}

\section{The Main Theorems}

The concept of pseudo-partition of a nest, introduced in \cite{Orr:MaIdNeAl},
will be technically useful and will provide useful terminology for our results.

\begin{definition}
  A \textit{pseudo-partition} of the nest $\N$ is a collection of pairwise
  orthogonal non-zero intervals which is maximal under set inclusion.
\end{definition}

We collect the following general properties of pseudo-partitions:

\begin{lemma}\label{extend-pseudo-partitions}
  Every collection of pairwise orthogonal intervals can be enlarged to form
  a pseudo-partition.
\end{lemma}

\begin{proof}
  Follows from a simple application of Zorn's Lemma.
\end{proof}

The Similarity Theorem for Nests \cite{Davidson:SiCoPeNeAl} says that two nests
are similar if and only if there is an order-dimension isomorphism between them.
That is to say, if there is a map $\theta:\N_1 \rightarrow \N_2$ which is
bijective, order-preserving, and such that
$\rank{\theta(N) - \theta(M)} = \rank{N-M}$
for all $N>M$ in $\N_1$.

\begin{lemma}\label{order-isomorphism}
  Let $\P$ be a pseudo-partition of $\N$ and let $P$ be the join of the
  intervals in $\P$. Then the map $\theta(N) :=  PN$ is an order-dimension
  isomorphism between $\N$ and $P\N$.
\end{lemma}

\begin{proof}
  To see that $\theta$ is rank-preserving, note
  that $N-M$ is finite rank if and only if it is a sum of finite rank atoms.
  In this case, since $P$ dominates all the atoms of $\N$,
  $(N-M)P=(N-M)$. If $N-M$ is not finite rank then either it dominates an
  infinite rank atom, or else it dominates infinitely many finite rank atoms,
  or else it dominates no atoms at all. In each of these cases the same is true
  of $(N-M)P$, which is therefore also infinite rank. Since $\theta$ preserves
  rank, in particular it is a bijection between $\N$ and $\N$.
\end{proof}

\begin{theorem}\label{belongs-to-max-ideal}
  Let $\N$ be a nest and $X \in \T{\N}$. Then $X$ belongs to a proper two-sided
  ideal of $\T{\N}$ if and only for every $a>0$ there is an interval $E$ of $\N$
  satisfying one of the following
  \begin{enumerate}
    \item $E$ contains no atoms and $\|EXE\| < a$,
    \item $E$ is an infinite rank atom and $\essnorm{EXE} < a$, or
    \item $E$ is a finite rank atom of rank $n$ and $\onenorm{EXE} < n a$.
  \end{enumerate}
\end{theorem}

\begin{proof}
  Note that each part of the condition involves a seminorm $\phi$ which, because
  compression to an interval is a homomorphism on $\T{\N}$, is
  submultiplicative in the following sense; $\phi(ABC) \le \|A\| \phi(B) \|C\|$.
  The sufficiency of the condition follows trivially from this because if $X$
  satisfies the condition then so does any finite sum $\sum A_iXB_i$. Thus
  everything in the ideal generated by $X$ satisfies the condition, but $I$
  clearly does not.

  To prove necessity we assume that $X$ fails the
  condition for a fixed value of $a > 0$. We split $\N$ into continuous and
  atomic parts and deal with these separately.
  Let $\I_a$
  and $\I_c$ be the two-sided ideals generated in $\T{\N_a}$ and $\T{\N_c}$
  respectively by $P_aXP_a$ and $P_cXP_c$. The result will follow when we show
  that $P_a \in \I_a$ and $P_c \in \I_c$.

  Note that if $S$ is a similarity which induces an order-dimension isomorphism
  $N \mapsto \tilde{N}$ from $\N$ to another nest $\tilde{\N}$ then
  \[
    (\tilde{N} - \tilde{M})SXS^{-1}(\tilde{N} - \tilde{M})
    = (\tilde{N} - \tilde{M})S(N-M)X(N-M)S^{-1}(\tilde{N} - \tilde{M})
  \]
  for all $N>M$ in
  $\N$. It follows that, for each seminorm $\phi$ in the condition of the theorem,
  $\phi(SXS^{-1}) \le K \phi(X)$, where $K$ is the condition number, and so
  if the condition of the theorem applies in $\T{\N}$, it also applies in all
  similar algebras.

  Thus we may replace $\N$ with a similar nest constructed as follows. Find a
  a pseudo-partition containing all the atoms of $\N$
  (Lemma~\ref{extend-pseudo-partitions}) and let $P$ be the join of its
  intervals. The map $N \mapsto PN$ is an order-dimension isomorphism
  (Lemma~\ref{order-isomorphism}) and so, by Davidson's Similarity Theorem,
  it is implemented by a similarity
  \cite{Davidson:SiCoPeNeAl}. Thus we may replace $\N$ by $P\N$,
  and so assume that $P_c$ is a sum of intervals.

  Thus $\|(N - M)P_cXP_c(N - M)\| \ge a$ for all $NP_c > MP_c$ and so in the
  context of $\T{P_c\N}$, $i_L(X) \ge a$ for all $L \in P_c\N$, where $i_L$ is
  the diagonal seminorm studied in \cite{Orr:TrAlIdNeAl}. Thus by
  \cite[Theorem 4.1]{Orr:TrAlIdNeAl}, there are $A, B \in \T{P_c\N}$
  such that $AXB = P_c$.

  In $\T{\N_a}$ let $\condex{X} := \sum EXE$ as $E$ ranges over all atoms
  of $\N_a$. We shall show that $\I_a$ contains an operator $Y$ with
  $\condex{Y} = I$ (the identity in $\T{\N_a}$ is of course $P_a$). Then let
  $\M$ be an atomic \masa\ in $\N_a' \subseteq \T{\N_a}$ and by Marcus,
  Spielman, and Srivastava's proof of the Kadison Singer Problem
  \cite{MarcusSielmanSrivastava:MiChPoKaSiPr},
  there are projections in $\M$ such that
  \[
    \left\| \sum^n_{i=1} P_i Y P_i - I \right\| <1
  \]
  It follows that $\sum P_i Y P_i$ is invertible in $\T{\N_a}$ and so $\I_a$
  contains $P_a$.

  It remains to prove, as we asserted above, that $\I_a$ contains an operator
  $Y$ with $\condex{Y} = I$. If $E$ is an infinite rank atom of $\N_a$ then
  $\essnorm{EXE} \ge a$ and so the spectral measure of $|EXE|$ on
  $(\frac{a}{2}, \infty)$ must be infinite rank (otherwise
  $\essnorm{EXE} \le \frac{a}{2})$. Thus there are $A, B$ in $E\T{\N_a}E$
  (which is identified with $\B(E\H)$)
  with $\|A\|, \|B\| \le \sqrt{\frac{2}{a}}$ and $AXB = E$. Because of the
  uniform norm bound, we can sum these operators to obtain $A$ and $B$ such that
  $\condex{AXB} = \sum' EAXBE = \sum' EAEXEBE = \sum' E$ where $\sum'$ is the
  sum is over all the infinite rank atoms $E$.

  Similarly, if $E$ is a finite rank atom then let $P_t$ be the spectral projection of
  $EXE$ on $(t, \infty)$. Since
  \[
    |EXE| \le t P_t^\perp + \|X\| P_t
  \]
  it follows
  \[
    n a \le \tr{|EXE|} \le t (n - \rank{P_t}) + \|X\| \rank{P_t}
  \]
  and so (for $0<t<a$)
  \[
    \frac{n}{\rank{P_t}} \le \frac{\|X\| - t}{a - t}
  \]
  By choosing $t$ small enough we can ensure
  \[
    \left\lceil\frac{n}{\rank{P_t}}\right\rceil
      \le K := \left\lceil\frac{2 \|X\|}{a}\right\rceil
  \]
  and this bound $K$ works simultaneously for all finite rank intervals.
  Thus for each finite rank interval $E$ we can find $A_1, \ldots, A_K$
  and $B_1, \ldots, B_K$ in $E\T{\N_a}E$ (which is identified with $M_n(\C)$),
  all with norm less than $1/\sqrt{t}$,
  such that $\sum A_i EXE B_i = E$. Because of the uniform bound on the norms
  and on the number of terms, we can sum over all finite rank atoms to get
  $A_1, \ldots, A_K$ and $B_1, \ldots, B_K$ in $\T{\N_a}$ such that
  \(
    \condex{\sum_1^K A_i EXE B_i}
      = \sum'' \sum_1^K EA_iXB_iE
      = \sum'' \sum_1^K EA_iEXEB_iE
      = \sum'' E
  \)
  where $\sum''$ is the sum over all the finite rank atoms $E$. Combining
  this with the result of the last paragraph we see that $\I_a$ contains
  an operator $Y$ with $\condex{Y} = \sum' E + \sum'' E = P_a$,
  the sum of all the atoms. By the remarks above, it follows from the
  Kadison-Singer Problem that $\I_a$ contains $P_a$ and we are done.
\end{proof}

\begin{definition}
  Let $\jmin$ be the set of operators $X \in \T{\N}$ with the property that for
  every $\e > 0$ there is a pseudo-partition $\P$ such that
  $\essnorm{EXE} < \e$ for all $E \in \P$ which contain an infinite rank atom,
  and such that $\| EXE \| < \e$ for all other $E \in \P$.
\end{definition}

Note that trivially it follows that every $X\in\jmin$ must satisfy $EXE=0$ for
all finite rank atoms and $EXE$ must be compact for all infinite rank atoms.

\begin{theorem}\label{strong-radical}
  Let $\T{\N}$ be a nest algebra. Then $\jmin$ is the strong radical of
  $\T{\N}$.
\end{theorem}

\begin{proof}
  First we shall show that $\jmin$ is contained in every maximal ideal of
  $\T{\N}$. Let $\J$ be a maximal ideal and suppose on the contrary that
  $\jmin \not\subseteq \J$. Then the algebraic sum $\J + \jmin$ contains $I$
  and so there are $J \in \J$ and $X \in \jmin$ such that $J = I - X$.
  We shall show that $J$ violates the criterion of
  Theorem~\ref{belongs-to-max-ideal} to belong to a proper ideal.

  Since $X \in \jmin$, find a pseudo-partition $\P$ such that
  $\essnorm{EXE} < 1/2$ when $E\in\P$ contains an infinite rank atom and
  $\|EXE\|< 1/2$ for the other $E\in\P$. Now, given any non-zero interval $E$,
  there must be an interval of $\P$ which has non-zero meet with $E$. Call this
  meet $E'$. Consider the following three cases:
  (i)~If $E$ contains no atoms then neither does $E'$, and so
    $\|EJE\| \ge \|E'XE'\| \ge \|E'\| - \|E'XE'\| \ge 1/2$.
  (ii)~If $E$ is an infinite rank atom then
    $\essnorm{EJE} \ge \essnorm{E} - \essnorm{EXE} = 1$
    (since $EXE$ is compact by the remarks following the definition of $\jmin$).
  (iii)~ If $E$ is a rank $n$ atom ($n$ finite) then $E = E'$ and
    $\onenorm{EJE} \ge \onenorm{E} - \onenorm{EXE} = n$
    (since $EXE = 0$, again by the remarks following the definition of $\jmin$).
  By Theorem~\ref{belongs-to-max-ideal} this contradicts the assumption that $J$
  belongs to a proper ideal.

  Next, we suppose that $X \not\in \jmin$ and we shall find a maximal ideal $\J$
  such that $X \not\in \J$. Consider first the behavior of $X$ on the atoms of
  $\N$. If there is a finite rank atom $E$ such that $EXE\not=0$ then
  $\{ Y\in\T{\N} : EYE = 0 \}$ is a maximal ideal which excludes $X$ and we are
  done. Likewise if $EXE$ is non-compact for some infinite rank atom $E$ then
  $\{ Y\in\T{\N} : \text{$EYE$ is compact} \}$ is a maximal ideal which excludes
  $X$. Thus we suppose that $EXE=0$ for all finite rank atoms and $EXE$ is
  compact for all infinite rank atoms.

  Since $X\not\in\jmin$, there is an $\e>0$ such that every pseudo-partition of
  $\N$ fails the condition from the definition of $\J$. Let $\P_a$ be a
  pseudo-partition which
  contains all the atoms and list its non-atomic intervals as $E_i$
  ($i=0,1,2,\ldots$). If it were possible to find a pseudo-partition $\P_i$ of
  each $E_i\N$ such that $\|EXE\|<\e$ for all $E\in\P_i$ then we could combine
  the intervals of the $\P_i$, together with the atoms of $\N$, to get a
  pseudo-partition of $\N$ which passes the condition from the definition of
  $\jmin$, contrary to the supposition for $\e$. Thus there is a non-atomic
  interval in $P_a$, say $E_0$, such that every pseudo-partition of $E_0\N$
  must contain an interval $E$ with $\|EXE\|\ge\e$. Since $E_0\N$ is continuous,
  by \cite[Proposition 4.1]{Orr:MaIdNeAl} $E_0XE_0$ is not in the strong
  radical of $\T{E_0\N}$. Let $\J_0$ be a maximal ideal of $\T{E_0\N}$ which
  excludes $E_0XE_0$. Then
  $\{ Y \in \T{\N} : E_0YE_0 \in \J_0\}$ is a maximal ideal of $\T{\N}$ which
  excludes $X$, and we are done.
\end{proof}

We shall use the characterization of the strong radical to give a
description of all the maximal two-sided ideals of a nest algebra in terms of
maximal ideals of the diagonal, and the description of maximal ideals of a
continuous next algebra from \cite{Orr:MaIdNeAl}. The result will follow from the
following lemma, which is a simple exercise in algebra:

\begin{lemma}\label{homomorphisms-and-maximal-ideals}
  Let $f:A \rightarrow B$ be a surjective homomorphism between unital algebras
  $A$ and $B$ and suppose $\ker{f} \subseteq R$ where $R$ is the strong radical
  of $A$. Then the maps
  $J \idealof A \mapsto f(J)$ and $J\idealof B \mapsto f^{-1}(J)$
  give a one-to-one correspondence between the maximal ideals of $A$ and $B$.
\end{lemma}

\begin{proof}
  If $J \idealof B$ then $f^{-1}(J) \idealof A$ and, since $f$ is surjective,
  if $J \idealof A$ then $f(J) \idealof B$. Also, because $f$ is surjective
  $f(f^{-1}(J)) = J$. Let $J$ be a maximal ideal of $A$. Clearly
  $f^{-1}(f(J)) \supseteq J$ and so is either $J$ or $A$. If it equals $A$ then
  $f(J) = f(f^{-1}(f(J))) = f(A)$. But then for each $a \in A$ there is a
  $j \in J$ such that $a - j \in \ker{f} \subseteq R \subseteq J$, which
  contradicts the fact $J$ is a proper ideal. Thus $f^{-1}(f(J)) = J$ when $J$
  is a maximal ideal of $A$.

  If $J$ is a maximal ideal of $A$ then $f(J)$ is a proper ideal of $B$
  (because $f^{-1}(f(J)) = J \subsetneq A$). If $K \idealof B$ satisfies
  $f(J) \subseteq K \subseteq B$ then
  $J = f^{-1}(f(J)) \subseteq f^{-1}(K) \subseteq A$ so that $f^{-1}(K)$
  is one of $J$ or $A$ and so $K = f(f^{-1}(K))$ is one of $f(J)$ or $B$. Hence
  $f(J)$ is a maximal ideal of $B$.

  Likewise if $J$ is a maximal ideal of $B$ then $f^{-1}(J)$ is a proper ideal
  of $A$ (because $f(f^{-1}(J) = J \subsetneq B$). Let $K$ be a maximal idea of
  $A$ which contains $f^{-1}(J)$. Thus $J = f(f^{-1}(J)) \subseteq f(K)$ and so
  $f(K)$ is either $J$ or $B$. If $f(K) = B$ then $K = f^{-1}(f(K)) = A$, which
  is a contradiction, and so $f(K) = J$. Thus $f^{-1}(J) = f^{-1}(f(K)) = K$,
  a maximal ideal.
\end{proof}

\begin{corollary}\label{characterize-the-max-ideals}
  Let $\P$ be a pseudo-partition of $\N$ which contains all the atoms and define
  $\Delta_a(X)$ to be the sum $\sum EXE$ as $E$ runs over the atoms of $\N$, and
  $\Delta_c(X)$ to be the sum $\sum EXE$ as $E$ runs over the remaining
  intervals of $\P$.
  Then every maximal ideal of $\T{\N}$ is of the form
  \[
    \{X \in \T{\N} : \Delta_a(X) \in \I_a\}
    \text { or }
    \{X \in \T{\N} : \Delta_c(x) \in \I_c\}
  \]
  where $\I_a, \I_c$ are maximal ideals of $\N_a'$ and $\T{\N_c}$ respectively.
\end{corollary}

\begin{proof}
  Write $\A := \T{\N}$.
  Since $\Delta := \Delta_a + \Delta_c$ is a surjective homomorphism of $\A$
  onto $\Delta(\A)$ and $\kernel{\Delta}$ is a subset of the strong radical of
  $\A$, it follows from Lemma~\ref{homomorphisms-and-maximal-ideals} that the
  maximal ideals of $\A$ are in one-to-one
  correspondence with the maximal ideals of $\Delta(\A)$. Moreover,
  $\Delta(\A) = \Delta_a(\A) \oplus \Delta_c(\A)$ and so maximal ideals of
  $\Delta(\A)$ are all of the form
  $\Delta_a(\A) \oplus \I_c$ and $\I_a \oplus \Delta_c(\A)$
  where $\I_a, \I_c$ are maximal ideals of $\Delta_a(\A)$ and $\Delta_c(\A)$
  respectively. Observe that $\Delta_a(\A) = \N_a'$. Finally, by slight abuse of
  notation, $\Delta_c$ is a surjective homomorphism from $\T{\N_c}$ onto
  $\Delta_c(\A)$ and its kernel is contained in the strong radical of
  $\T{\N_c}$. So again by Lemma~\ref{homomorphisms-and-maximal-ideals} the
  maximal ideals of $\T{\N_c}$ are also in one-to-one
  correspondence with the maximal ideals of $\Delta_c(\A)$.
\end{proof}

\begin{example}

  If $\N$ is an atomic nest then $\J_0 := \ker{\Delta_a}$ is the set of all
  operators in $\T{\N}$ which vanish on the atoms.
  By Corollary~\ref{characterize-the-max-ideals}, the maximal ideals are
  precisely the sets of the form $\J + \J_0$ where $\J$ is a maximal ideal of
  the $C^*$-algebra $\N'$. In particular if $\T{\N}$ is the algebra of upper
  triangular matrices wrt the standard basis on $l^2(\NN)$, then the maximal
  ideals are precisely the sets of the form $\J + \S$, where $\S$ is the set
  of strictly upper triangular matrices and $\J$ is a maximal ideal of the
  diagonal algebra, identified with $l^\infty(\NN)$.
\end{example}

\section{Notes on Epimorphisms}

In \cite{DavidsonHarrisonOrr:EpNeAl} we studied the epimorphisms of one nest
algebra onto another. We showed that all such epimorphisms are continuous
and we described their possible forms.
However our characterization had one gap: we
were unable to rule out the possibility that there might exist an epimorphism
from an atomic nest algebra with all atoms finite rank onto either $\bh$ or a
continuous nest algebra.

We conjectured
\cite[Conjecture 8.1]{DavidsonHarrisonOrr:EpNeAl} that no such map exists, but
were unable to settle the question. However in the discussion following that
conjecture we showed that if such a map $\phi$ were to exist, its kernel would
contain an operator of the form $I+X$ where $\Delta_a(X) = 0$. In view of
Theorem~\ref{belongs-to-max-ideal} we can now conclude that this forces $\phi$
to be the zero map. Thus no epimorphism can exist from an atomic
nest algebra with finite rank atoms onto $\bh$ or a continuous nest, and the
cases described in \cite{DavidsonHarrisonOrr:EpNeAl} thus provide a complete
picture of all possible epimorphisms between nest algebras.

\bibliography{bibliography}
\bibliographystyle{plain}

\end{document}